\documentclass[a4paper,twoside]{amsart}

\usepackage{amssymb}
\usepackage{amsmath,amssymb,epsf,wrapfig,epic,eepic,trig,mathrsfs,dsfont,color,graphicx,accents}
\usepackage[abs]{overpic}
\usepackage{pdfpages}

\theoremstyle{plain}
\newtheorem{thm}{Theorem}[section]
\newtheorem{prop}[thm]{Proposition}
\newtheorem{cor}[thm]{Corollary}
\newtheorem{lem}[thm]{Lemma}

\theoremstyle{remark}
\newtheorem{rmk}[thm]{Remark}

\theoremstyle{definition}

\newtheorem{exam}[thm]{Example}

\AtBeginDvi{} 
\usepackage{graphics}
\usepackage{epsfig} 
\usepackage[english]{babel}

\usepackage{epsf,color,fancybox,colortbl,multicol,graphicx,amsmath,wrapfig,epic,eepic, pb-diagram}

\def\rmoveio#1#2{%#2=2 oriented #2=1 non-oriented
\setlength{\unitlength}{#1}
\begin{picture}(50,30)
\put(5,0){\line(0,1){30}}

{\allinethickness{.8pt}
\put(10,15){\vector(1,0){13}}
\put(23,15){\vector(-1,0){13}}}

\qbezier(25,0)(25,20)(40,20)
\qbezier(40,20)(45,20)(45,15)
\qbezier(45,15)(45,10)(40,10)
\qbezier(40,10)(35,10)(31,14)
\qbezier(28,17)(25,25)(25,30)

\ifnum#2=2
\put(3,28){\path(0,0)(2,2)(4,0)}
\put(28,28){\path(0,0)(2,2)(4,0)}
\put(38,15){\makebox{${\Huge c_{1}}$}}
\fi

\end{picture}
}

\def\rmoveiio#1#2{%#2=2 oriented #2=1 non-oriented #2=3 cohearent oriented
\setlength{\unitlength}{#1}
\begin{picture}(60,30)
\put(5,0){\line(0,1){30}}
\put(15,0){\line(0,1){30}}

{\allinethickness{.8pt}
\put(20,15){\vector(1,0){15}}
\put(35,15){\vector(-1,0){15}}}

\qbezier(40,0)(42,1)(47,3)
\qbezier(52,6)(68,15)(52,24)
\qbezier(47,27)(42,30)(40,30)

\qbezier(60,0)(20,15)(60,30)

\ifnum#2=2
\put(2,27){\path(0,0)(3,3)(6,0)}
\put(12,27){\path(0,0)(3,3)(6,0)}
\put(40,27){\path(0,0)(0,3)(3,3)}
\put(60,27){\path(0,0)(0,3)(-3,3)}
\put(47,16){\makebox{${\Huge c_{1}}$}}
\put(47,10){\makebox{${\Huge c_{2}}$}}
\fi

\ifnum#2=3
\put(2,2){\path(0,0)(3,-3)(6,0)}
\put(12,27){\path(0,0)(3,3)(6,0)}
\put(40,3){\path(0,0)(0,-3)(3,-3)}
\put(60,27){\path(0,0)(0,3)(-3,3)}
\put(47,16){\makebox{${\Huge c_{1}}$}}
\put(47,10){\makebox{${\Huge c_{2}}$}}
\fi

\end{picture}
}
\def\rmoveiiio#1#2{%#2=2 oriented #2=1 non-oriented #2=3 cohearent oriented
\setlength{\unitlength}{#1}
\begin{picture}(75,30)
\put(0,0){\line(1,1){15}}
\qbezier(15,15)(20,20)(20,30)

\put(10,0){\line(-1,1){4}}
\qbezier(4,6)(-5,15)(5,25)
\put(5,25){\line(1,1){5}}

\qbezier(20,0)(20,10)(16,14)
\put(14,16){\line(-1,1){8}}
\put(4,26){\line(-1,1){4}}

{\allinethickness{.8pt}
\put(25,15){\vector(1,0){15}}
\put(40,15){\vector(-1,0){15}}}

\qbezier(50,0)(50,10)(55,15)
\put(55,15){\line(1,1){15}}

\put(60,0){\line(1,1){5}}
\qbezier(65,5)(75,15)(66,24)
\put(64,26){\line(-1,1){4}}

\put(70,0){\line(-1,1){4}}
\put(64,6){\line(-1,1){8}}
\qbezier(54,16)(50,20)(50,30)

\ifnum#2=2
\put(0,27){\path(0,0)(0,3)(3,3)}
\put(7,30){\path(0,0)(3,0)(3,-3)}
\put(17,27){\path(0,0)(3,3)(6,0)}

\put(10,23){\makebox{${\Huge c_{1}}$}}
\put(10,3){\makebox{${\Huge c_{2}}$}}
\put(20,13){\makebox{${\Huge c_{3}}$}}

\put(47,27){\path(0,0)(3,3)(6,0)}
\put(60,27){\path(0,0)(0,3)(3,3)}
\put(67,30){\path(0,0)(3,0)(3,-3)}

\put(70,23){\makebox{${\Huge c'_{2}}$}}
\put(70,3){\makebox{${\Huge c'_{1}}$}}
\put(60,13){\makebox{${\Huge c'_{3}}$}}
\fi

\end{picture}
}
\def\rmovevio#1#2{%#2=2 oriented #2=1 non-oriented
\setlength{\unitlength}{#1}
\begin{picture}(50,30)
\put(5,0){\line(0,1){30}}

{\allinethickness{.8pt}
\put(10,15){\vector(1,0){15}}
\put(25,15){\vector(-1,0){15}}}

\qbezier(30,0)(30,20)(45,20)
\qbezier(45,20)(50,20)(50,15)
\qbezier(50,15)(50,10)(45,10)
\qbezier(45,10)(30,10)(30,30)

\put(34,15){\circle{5}}

\ifnum#2=2
\put(3,28){\path(0,0)(2,2)(4,0)}
\put(28,28){\path(0,0)(2,2)(4,0)}
\put(38,15){\makebox{${\Huge c_{1}}$}}
\fi

\end{picture}
}

\def\rmoveviio#1#2{%#2=2 oriented #2=1 non-oriented #2=3 cohearent oriented
\setlength{\unitlength}{#1}
\begin{picture}(60,30)
\put(5,0){\line(0,1){30}}
\put(15,0){\line(0,1){30}}

{\allinethickness{.8pt}
\put(20,15){\vector(1,0){15}}
\put(35,15){\vector(-1,0){15}}}

\qbezier(40,0)(80,15)(40,30)
\qbezier(60,0)(20,15)(60,30)
\put(50,4){\circle{5}}
\put(50,25){\circle{5}}

\ifnum#2=2
\put(2,27){\path(0,0)(3,3)(6,0)}
\put(12,27){\path(0,0)(3,3)(6,0)}
\put(40,27){\path(0,0)(0,3)(3,3)}
\put(60,27){\path(0,0)(0,3)(-3,3)}
\put(47,15){\makebox{${\Huge c_{1}}$}}
\put(47,10){\makebox{${\Huge c_{2}}$}}
\fi

\ifnum#2=3
\put(2,2){\path(0,0)(3,-3)(6,0)}
\put(12,27){\path(0,0)(3,3)(6,0)}
\put(40,3){\path(0,0)(0,-3)(3,-3)}
\put(60,27){\path(0,0)(0,3)(-3,3)}
\put(47,15){\makebox{${\Huge c_{1}}$}}
\put(47,10){\makebox{${\Huge c_{2}}$}}
\fi

\end{picture}
}

\def\rmoveviiio#1#2{%#2=2 oriented #2=1 non-oriented #2=3 cohearent oriented
\setlength{\unitlength}{#1}
\begin{picture}(70,30)
\put(0,0){\line(1,1){15}}
\qbezier(15,15)(20,20)(20,30)

\put(10,0){\line(-1,1){5}}
\qbezier(5,5)(-5,15)(5,25)
\put(5,25){\line(1,1){5}}

\qbezier(20,0)(20,10)(15,15)
\put(15,15){\line(-1,1){15}}
\put(5,5){\circle{5}}
\put(15,15){\circle{5}}
\put(5,25){\circle{5}}

{\allinethickness{.8pt}
\put(28,15){\vector(1,0){14}}
\put(42,15){\vector(-1,0){14}}}

\qbezier(50,0)(50,10)(55,15)
\put(55,15){\line(1,1){15}}

\put(60,0){\line(1,1){5}}
\qbezier(65,5)(75,15)(65,25)
\put(65,25){\line(-1,1){5}}

\put(70,0){\line(-1,1){15}}
\qbezier(55,15)(50,20)(50,30)

\put(65,5){\circle{5}}
\put(55,15){\circle{5}}
\put(65,25){\circle{5}}

\ifnum#2=2
\put(0,27){\path(0,0)(0,3)(3,3)}
\put(7,30){\path(0,0)(3,0)(3,-3)}
\put(17,27){\path(0,0)(3,3)(6,0)}

\put(20,13){\makebox{${\Huge c_{1}}$}}

\put(47,27){\path(0,0)(3,3)(6,0)}
\put(60,27){\path(0,0)(0,3)(3,3)}
\put(67,30){\path(0,0)(3,0)(3,-3)}

\put(60,13){\makebox{${\Huge c'_{1}}$}}
\fi

\end{picture}
}

\def\rmovevivo#1#2{%#2=2 oriented #2=1 non-oriented #2=3 cohearent oriented
\setlength{\unitlength}{#1}
\begin{picture}(70,30)
\put(0,0){\line(1,1){15}}
\qbezier(15,15)(20,20)(20,30)

\put(10,0){\line(-1,1){5}}
\qbezier(5,5)(-5,15)(5,25)
\put(5,25){\line(1,1){5}}

\qbezier(20,0)(20,10)(16,14)
\put(14,16){\line(-1,1){14}}
\put(5,5){\circle{5}}
\put(5,25){\circle{5}}

{\allinethickness{.8pt}
\put(28,15){\vector(1,0){14}}
\put(42,15){\vector(-1,0){14}}}

\qbezier(50,0)(50,10)(55,15)
\put(55,15){\line(1,1){15}}

\put(60,0){\line(1,1){5}}
\qbezier(65,5)(75,15)(65,25)
\put(65,25){\line(-1,1){5}}

\put(70,0){\line(-1,1){14}}
\qbezier(54,16)(50,20)(50,30)

\put(65,5){\circle{5}}
\put(65,25){\circle{5}}

\ifnum#2=2
\put(0,27){\path(0,0)(0,3)(3,3)}
\put(7,30){\path(0,0)(3,0)(3,-3)}
\put(17,27){\path(0,0)(3,3)(6,0)}

\put(20,13){\makebox{${\Huge c_{1}}$}}

\put(47,27){\path(0,0)(3,3)(6,0)}
\put(60,27){\path(0,0)(0,3)(3,3)}
\put(67,30){\path(0,0)(3,0)(3,-3)}

\put(60,13){\makebox{${\Huge c'_{1}}$}}
\fi

\end{picture}
}
\def\fmove#1#2{%#2=2 oriented #2=1 non-oriented #2=3 cohearent oriented
\setlength{\unitlength}{#1}
\begin{picture}(75,30)

\put(0,0){\line(1,1){15}}
\qbezier(15,15)(20,20)(20,30)

\put(10,0){\line(-1,1){4}}
\qbezier(4,6)(-5,15)(5,25)
\put(5,25){\line(1,1){5}}

\qbezier(20,0)(20,10)(16,14)
\put(14,16){\line(-1,1){14}}
\put(5,25){\circle{5}}

{\allinethickness{.8pt}
\put(25,15){\vector(1,0){15}}
\put(40,15){\vector(-1,0){15}}}

\qbezier(50,0)(50,10)(55,15)
\put(55,15){\line(1,1){15}}

\put(60,0){\line(1,1){5}}
\qbezier(65,5)(75,15)(66,24)
\put(64,26){\line(-1,1){4}}

\put(70,0){\line(-1,1){14}}
\qbezier(54,16)(50,20)(50,30)
\put(65,5){\circle{5}}

\ifnum#2=2
\put(0,27){\path(0,0)(0,3)(3,3)}
\put(7,30){\path(0,0)(3,0)(3,-3)}
\put(17,27){\path(0,0)(3,3)(6,0)}

\put(10,23){\makebox{${\Huge c_{1}}$}}
\put(10,3){\makebox{${\Huge c_{2}}$}}
\put(20,13){\makebox{${\Huge c_{3}}$}}

\put(47,27){\path(0,0)(3,3)(6,0)}
\put(60,27){\path(0,0)(0,3)(3,3)}
\put(67,30){\path(0,0)(3,0)(3,-3)}

\put(70,23){\makebox{${\Huge c'_{2}}$}}
\put(70,3){\makebox{${\Huge c'_{1}}$}}
\put(60,13){\makebox{${\Huge c'_{3}}$}}
\fi

\end{picture}
}

\def\canocutsysi#1#2{%#2=2 with lavel #2=1 non label
\setlength{\unitlength}{#1}
\begin{picture}(35,20)

\put(5,0){\line(1,1){20}}
\put(25,0){\line(-1,1){20}}
\put(15,10){\circle{4}}

\put(10,5){\path(0,0)(-3,0)(0,-3)(0,0)}
\put(22,3){\path(0,0)(-3,0)(0,3)(0,0)}

\put(5,0){\path(0,3)(0,0)(3,0)}
\put(25,0){\path(-3,0)(0,0)(0,3)}

\ifnum#2=2
\put(0,15){\makebox{{$i$}}}
\put(0,0){\makebox{{$i$}}}
\put(27,15){\makebox{{$i+1$}}}
\put(27,0){\makebox{{$i+1$}}}
\fi

\end{picture}
}

\def\alexnum#1{%#2=2 oriented #2=1 non-oriented #2=3 cohearent oriented
\setlength{\unitlength}{#1}
\begin{picture}(90,20)

\put(5,0){\line(1,1){20}}
\put(25,0){\line(-1,1){9}}
\put(14,11){\line(-1,1){9}}

\put(5,0){\path(0,3)(0,0)(3,0)}
\put(25,0){\path(-3,0)(0,0)(0,3)}

\put(0,15){\makebox{{$i$}}}
\put(0,0){\makebox{{$i$}}}
\put(27,15){\makebox{{$i+1$}}}
\put(27,0){\makebox{{$i+1$}}}

\put(55,0){\line(1,1){9}}
\put(75,0){\line(-1,1){20}}
\put(66,11){\line(1,1){9}}

\put(55,0){\path(0,3)(0,0)(3,0)}
\put(75,0){\path(-3,0)(0,0)(0,3)}

\put(50,15){\makebox{{$i$}}}
\put(50,0){\makebox{{$i$}}}
\put(77,15){\makebox{{$i+1$}}}
\put(77,0){\makebox{{$i+1$}}}

\end{picture}
}

\def\alexnumv#1{%#2=2 oriented #2=1 non-oriented #2=3 cohearent oriented
\setlength{\unitlength}{#1}
\begin{picture}(35,20)

\put(5,0){\line(1,1){20}}
\put(25,0){\line(-1,1){20}}
\put(15,10){\circle{4}}

\put(5,0){\path(0,3)(0,0)(3,0)}
\put(25,0){\path(-3,0)(0,0)(0,3)}

\put(0,15){\makebox{{$j$}}}
\put(0,0){\makebox{{$i$}}}
\put(27,15){\makebox{{$i$}}}
\put(27,0){\makebox{{$j$}}}

\end{picture}
}

\def\rmovetiiio#1#2{%#2=2 oriented #2=1 non-oriented #2=3 cohearent oriented
\setlength{\unitlength}{#1}
\begin{picture}(70,20)

\put(0,0){\line(1,1){20}}
\put(20,0){\line(-1,1){9}}
\put(9,11){\line(-1,1){9}}

\put(7,3){\line(-1,1){4}}
\put(13,3){\line(1,1){4}}
\put(3,13){\line(1,1){4}}
\put(17,13){\line(-1,1){4}}

{\allinethickness{.8pt}
\put(25,10){\vector(1,0){10}}
\put(35,10){\vector(-1,0){10}}}

\qbezier(40,5)(48,20)(54,11)
\qbezier(56, 9)(62,0)(70,15)

\qbezier(40,15)(48,0)(55,10)
\qbezier(55,10)(62,20)(70,5)

\put(43,10){\circle{5}}
\put(67,10){\circle{5}}

\ifnum#2=2
\put(0,17){\path(0,0)(0,3)(3,3)}
\put(17,20){\path(0,0)(3,0)(3,-3)}

\put(15,8){\makebox{${\Huge c_{1}}$}}

\put(40,12){\path(0,0)(0,3)(3,3)}
\put(67,15){\path(0,0)(3,0)(3,-3)}

\put(53,3){\makebox{${\Huge c'_{1}}$}}
\fi

\end{picture}
}
\def\alexnum#1{%#2=2 oriented #2=1 non-oriented #2=3 cohearent oriented
\setlength{\unitlength}{#1}
\begin{picture}(90,20)

\put(5,0){\line(1,1){20}}
\put(25,0){\line(-1,1){9}}
\put(14,11){\line(-1,1){9}}

\put(5,0){\path(0,3)(0,0)(3,0)}
\put(25,0){\path(-3,0)(0,0)(0,3)}

\put(0,15){\makebox{{$i$}}}
\put(0,0){\makebox{{$i$}}}
\put(27,15){\makebox{{$i+1$}}}
\put(27,0){\makebox{{$i+1$}}}

\put(55,0){\line(1,1){9}}
\put(75,0){\line(-1,1){20}}
\put(66,11){\line(1,1){9}}

\put(55,0){\path(0,3)(0,0)(3,0)}
\put(75,0){\path(-3,0)(0,0)(0,3)}

\put(50,15){\makebox{{$i$}}}
\put(50,0){\makebox{{$i$}}}
\put(77,15){\makebox{{$i+1$}}}
\put(77,0){\makebox{{$i+1$}}}

\end{picture}
}

\begin{document}

\title{Parallelization of Welded Links
}
\author[Naoko Kamada]{Naoko Kamada}
\address{Graduate School of Natural Sciences,  Nagoya City University,
Mizuho-ku, Nagoya, Aichi 467-8501, Japan}
\email{kamada@nsc.nagoya-cu.ac.jp}

%\author[Seiichi Kamada]{Seiichi Kamada}
%\address{Department of Mathematics,
%Osaka City University,
%Sugimoto, Sumiyoshi-ku,
%Osaka 558-8585, Japan}
%\email{skamada@sci.osaka-cu.ac.jp}

\author[Seiichi Kamada]{Seiichi Kamada}
\address{Department of Mathematics,
The University of Osaka,
Toyonaka, Osaka 560-0043, Japan}
\email{kamada@math.sci.osaka-u.ac.jp}

%\subjclass[2010]{57M25}
\subjclass[2020]{57K12}

%\keywords{Welded links, Virtual links, Parallel diagrams, Parallelization}
\keywords{welded knots; virtual knots; parallel diagrams; parallelization, quandles}
\date{\today}

\begin{abstract}
The notion of a welded link was introduced by Fenn, Rim\'anyi, and Rourke as an analogue of welded braids.
A welded link is defined as an equivalence class of link diagrams that may contain virtual crossings, 
where the equivalence is generated by the classical and virtual Reidemeister moves together with the welded moves.
In this paper, we introduce a parallelization construction for welded link diagrams and show that it is well defined: if two diagrams represent equivalent welded links, then the corresponding parallel diagrams obtained by our construction are also equivalent.
When the two parallel strands are given parallel orientations, the resulting diagram admits a checkerboard coloring, whereas if they are assigned opposite orientations, the diagram is almost classical.
Our construction further yields a decomposition in which one component is a copy of the original diagram and the other is a diagram representing a trivial welded link.
We also investigate quandle colorings and the fundamental quandle of the parallel diagram, deriving a presentation from that of the original diagram.
Finally, we examine conditions under which the parallel diagram is non-split. 
\end{abstract}
\maketitle

%%%%%%%%%%%%%%%%%%%%%
\section{Introduction}\label{sect:intro} 

The notion of a welded link was first introduced by Fenn, Rim\'anyi, and Rourke \cite{FRR} as an analogue of welded braids.
A welded link is defined as an equivalence class of link diagrams that may contain virtual crossings, where the equivalence is generated by the Reidemeister moves, the virtual Reidemeister moves, and the welded moves, which are illustrated in Figures~\ref{fig:rmove}, \ref{fig:rmovev}, and \ref{fig:fmove}.

Throughout this paper, we refer to link diagrams that may contain virtual crossings as virtual link diagrams, or simply diagrams.
If we do not allow welded moves, then the equivalence relation on diagrams generated only by the Reidemeister moves and the virtual Reidemeister moves yields the notion of a virtual link.
It is well known that for diagrams without virtual crossings (i.e., classical link diagrams), the three notions—equivalence as classical links, equivalence as virtual links, and equivalence as welded links—coincide (cf. \cite{rGPV, rkauD}).
In this sense, both virtual knot theory and welded knot theory can be regarded as generalizations of classical knot theory.
Moreover, virtual knots correspond to abstract knots \cite{rkk} and stable equivalence classes of knots on surfaces \cite{rCKS}, whereas welded knots are closely related to ribbon torus-knots in the 4-space \cite{rSatoh}.

For a virtual link diagram, several notions of a parallel diagram have been introduced and studied (cf. \cite{rkn2, rkn3}).
Most previous works have been carried out from the viewpoint of virtual links, and a typical requirement has been that if two diagrams are equivalent as virtual links, then their parallel diagrams are also equivalent as virtual links.

In this paper, we introduce a method for constructing a parallel diagram from a given diagram and show that if two original diagrams are equivalent as welded links, then the parallel diagrams obtained by our method are also equivalent.

When an orientation is assigned to the parallel diagram produced by our method so that the two parallel arcs have parallel orientations, the resulting diagram admits a checkerboard coloring.
On the other hand, when the two parallel arcs are given opposite (i.e., non-parallel) orientations, we show that the diagram is almost classical.

Moreover, the parallel diagram obtained by our method has the following feature: with respect to the orientation of the original diagram, the subdiagram appearing on the right-hand side can always be regarded as a copy of the original diagram, whereas the subdiagram on the left-hand side represents a trivial welded link. 

We also discuss quandle colorings and the fundamental quandle of parallel diagrams.
In particular, from a presentation of the fundamental quandle of the original diagram, we derive a presentation of the fundamental quandle of the parallel diagram.

In the final section, we consider when the parallel diagram is non-split. 

\begin{figure}[h]
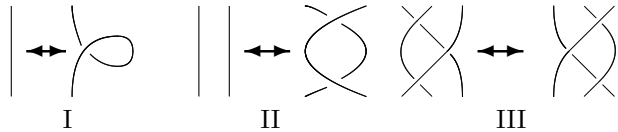

\centerline{
\begin{tabular}{ccc}
\rmoveio{.4mm}{1}&\rmoveiio{.4mm}{1}&\rmoveiiio{.4mm}{1}\\
I &  II & III  \\
\end{tabular}}
\caption{Reidemeister moves}
\label{fig:rmove}
\end{figure}

\begin{figure}[h]
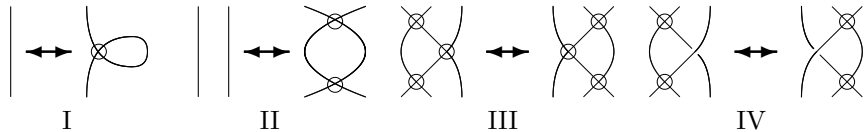

\centerline{
\begin{tabular}{cccc}
\rmovevio{.4mm}{1}&\rmoveviio{.4mm}{1}&\rmoveviiio{.4mm}{1}&\rmovevivo{.4mm}{1}\\
I &  II & III &IV \\
\end{tabular}}
\caption{Virtual Reidemeister moves}
\label{fig:rmovev}
\end{figure}

\begin{figure}[h]
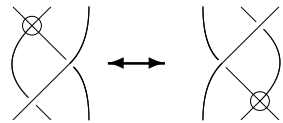

\centerline{\begin{tabular}{c}
\fmove{.5mm}{1}\\
\end{tabular}}
\caption{Forbidden move}
\label{fig:fmove}
\end{figure}

%%%%%%%%%%%%%%%%%%%%%
\section{Construction}\label{sect:const} 

Let $D$ be a virtual link diagram.  
Let $\phi_+(D)$ be a virtual link diagram such that for each classical or virtual crossing of $D$ we consider a diagram as in Figure~\ref{fig:phipositive}, and in the outside of regular neighborhoods of the crossings of $D$ we consider just parallel copies of the arcs.  For example, see Figure~\ref{fig:phiexample}. 

% $\phi_+$ 
\begin{figure}[h]
\centerline{
\includegraphics[width=12cm]{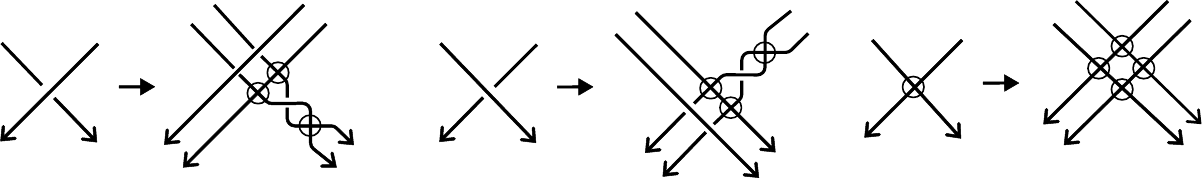}}
\centerline{positive crossing\phantom{MMMMM} negative crossing\phantom{MMMMM}  virtual crossing}
\caption{The construction $\phi_+$} 
\label{fig:phipositive}
\end{figure}

\begin{figure}[h]
\centerline{
\includegraphics[width=10cm]{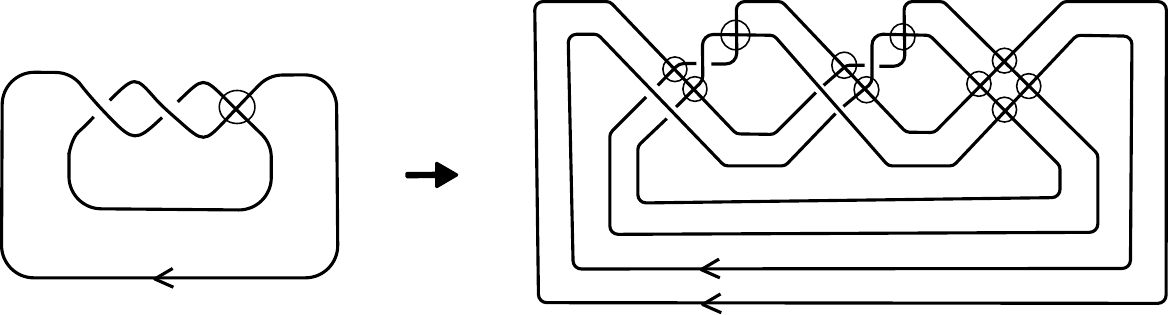}}
\centerline{$D_1$\phantom{MMMMMMMMMMMMMMM} $\phi_+ (D_1)$\phantom{MMM}}
\end{figure}

\begin{figure}[h]
\centerline{
\includegraphics[width=10cm]{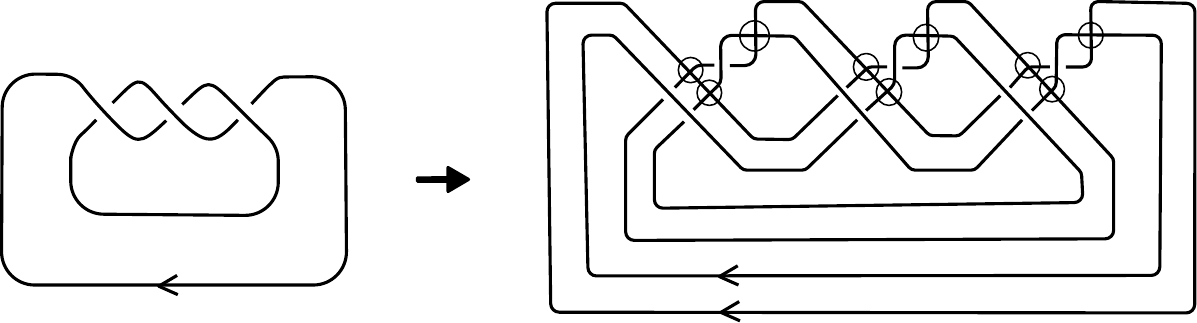}}
\centerline{$D_2$\phantom{MMMMMMMMMMMMMMM} $\phi_+ (D_2)$\phantom{MMM}}
\caption{Examples of $\phi_+(D)$} 
\label{fig:phiexample}
\end{figure}

Similarly, for a virtual link diagram $D$, we consider a diagram $\phi_-(D)$ constructed by using diagrams depicted in Figure~\ref{fig:phinegative}.

% $\phi_-$
\begin{figure}[h]
\centerline{ 
\includegraphics[width=12cm]{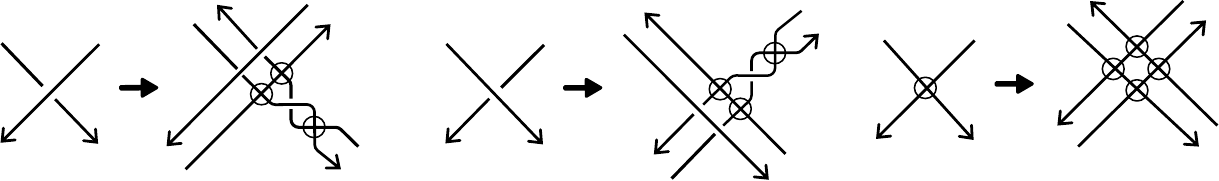}}
\centerline{positive crossing\phantom{MMMMM} negative crossing\phantom{MMMMM}  virtual crossing}
\caption{The construction $\phi_-$} 
\label{fig:phinegative}
\end{figure}

In the parallel diagram $\phi_+(D)$ obtained from a diagram $D$, we call the subdiagram that appears on the right-hand side with respect to the orientation of $D$ the \emph{main subdiagram} of $\phi_+(D)$, and denote it by $\phi_+(D)_{\rm R}$. 

Similarly, the subdiagram of $\phi_+(D)$ that appears on the left-hand side with respect to the orientation of $D$ is called the \emph{secondary subdiagram} of $\phi_+(D)$, and is denoted by $\phi_+(D)_{\rm L}$. (Here, R and L stand for \lq\lq Right\rq\rq and \lq\lq Left,\rq\rq respectively.)

The main subdiagram $\phi_+(D)_{\rm R}$ can always be regarded as a copy of the original diagram $D$, and $D$ and $\phi_+(D)_{\rm R}$ represent the same welded link (indeed, the same virtual link).
On the other hand, the secondary subdiagram $\phi_+(D)_{\rm L}$ contains only virtual crossings, and hence $\phi_+(D)_{\rm L}$ represents a trivial welded link (in fact, a trivial virtual link).

The same observations apply to the parallel diagram $\phi_-(D)$. 
We call the subdiagram of $\phi_-(D)$ that appears on the right side with respect to the orientation of $D$ the \emph{main subdiagram} of $\phi_-(D)$, and denote it by $\phi_-(D)_{\rm R}$.
Similarly, the subdiagram of $\phi_-(D)$ that appears on the left side with respect to the orientation of $D$ is called the \emph{secondary subdiagram} of $\phi_-(D)$, and is denoted by $\phi_-(D)_{\rm L}$. 

The only difference between the parallel diagrams $\phi_+(D)$ and $\phi_-(D)$ lies in the orientations of their secondary subdiagrams.
The secondary subdiagram of $\phi_+(D)$ is oriented parallel to the orientation of the original diagram $D$, whereas the secondary subdiagram of $\phi_-(D)$ is given the opposite orientation.

Now we have two maps 
\[
\phi_+ : \{ \text{virtual link diagrams} \} \rightarrow \{ \text{virtual link diagrams} \}
\]
and 
\[
\phi_- : \{ \text{virtual link diagrams} \} \rightarrow \{ \text{virtual link diagrams} \}. 
\]

\begin{thm}\label{thm:mainA}
If two virtual link diagrams $D$ and $D'$ are equivalent as welded links, then 
$\phi_+(D)$ and $\phi_+(D')$ are equivalent as welded links, and so are $\phi_-(D)$ and $\phi_-(D')$. 
\end{thm}

Thus, the maps $\phi_+$ and $\phi_+$ induce maps from the set of welded links to itself, which we use the same symbols $\phi_+$ and $\phi_+$ to denote:  
Now we have two maps 
\[
\phi_+ : \{ \text{welded links} \} \rightarrow \{ \text{welded links} \}
\]
and 
\[
\phi_- : \{ \text{welded links} \} \rightarrow \{ \text{welded links} \}. 
\]

%%%%%%%%%%%%%%%%%%%%%

\section{Proof of Theorem~\ref{thm:mainA}}\label{sect:proofoftheorem}

In this section, we prove Theorem~\ref{thm:mainA}.

We call the tangle diagrams depicted in Figure~\ref{fig:twists}
a $-1v$-twist and a $v1$-twist, respectively.

\begin{figure}[h]
\centerline{
\includegraphics[width=5.5cm]{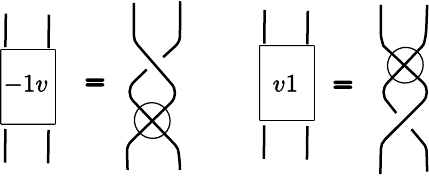}
}
\centerline{$-1v$-twist\phantom{MMMM}$v1$-twist}
\caption{A $-1v$-twist (left) and a $v1$-twist (right)}
\label{fig:twists}
\end{figure}

\begin{lem}\label{lem:twistpass}
The two local moves depicted in Figure~\ref{fig:twistspass}
do not change the equivalence classes of virtual link diagrams as welded links,
where the box in the figure stands for either a $-1v$-twist or a $v1$-twist.
\end{lem}

\begin{figure}[h]
\centerline{
\includegraphics[width=10cm]{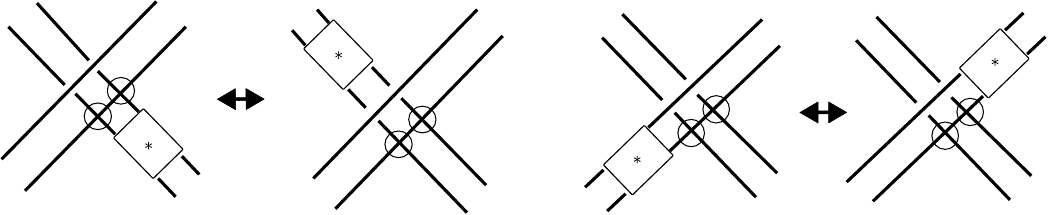}}
\caption{Passing twists through arcs}
\label{fig:twistspass}
\end{figure}

\begin{proof}
The moves depicted in Figures~\ref{fig:movev} and \ref{fig:moveo}
do not change the equivalence classes as welded links.
Hence, the move shown on the left side of Figure~\ref{fig:twistspass}
does not change the welded link type.
For the move on the right side, the case in which the box represents a $v1$-twist
is shown in Figure~\ref{fig:proof1}.
The case of a $-1v$-twist is analogous.
\end{proof}

\begin{figure}[h]
\centerline{
\includegraphics[width=3cm]{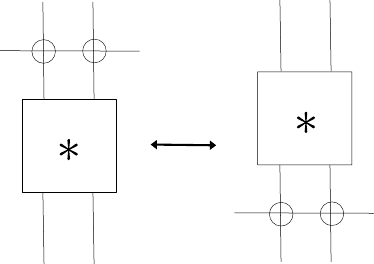}
}
\caption{}
\label{fig:movev}
\end{figure}

\begin{figure}[h]
\centerline{
\includegraphics[width=3cm]{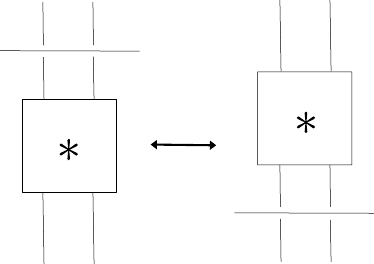}
}
\caption{}
\label{fig:moveo}
\end{figure}

\begin{figure}[h]
\centerline{
\includegraphics[width=8cm]{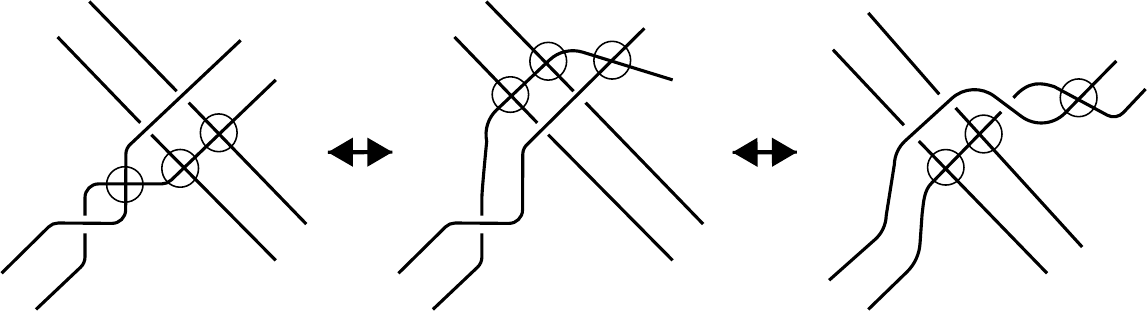}
}
\caption{}
\label{fig:proof1}
\end{figure}

By this lemma, we may slide a $-1v$-twist or a $v1$-twist to any position along the parallel copy of the original diagram without changing the welded link type.
For example, the two diagrams in Figure~\ref{fig:exequiv1} represent the same welded link.

\begin{figure}[h]
\centerline{
\includegraphics[width=8cm]{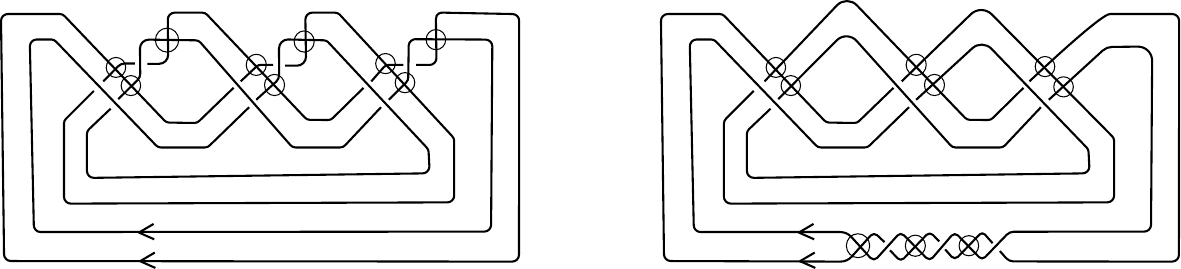}
}
\caption{}
\label{fig:exequiv1}
\end{figure}

\begin{proof}[Proof of Theorem~\ref{thm:mainA}]
It is shown in Theorem~1.2 of \cite{rpolyak} that
any oriented Reidemeister move can be realized as a sequence of the five oriented Reidemeister moves
$\Omega1a$, $\Omega1c$, $\Omega2c$, $\Omega2d$, and $\Omega3b$.
We show that if $D$ and $D'$ are related by any of these moves, then
$\phi_+(D)$ and $\phi_+(D')$ are equivalent as welded links.
Figure~\ref{fig:proofR1} illustrates the cases of $\Omega1a$ and $\Omega1c$,
Figure~\ref{fig:proofR2} the cases of $\Omega2c$ and $\Omega2d$,
and Figure~\ref{fig:proofR3} the case of $\Omega3b$,
where Lemma~\ref{lem:twistpass} is used.

\begin{figure}[h]
\centerline{
\includegraphics[width=8cm]{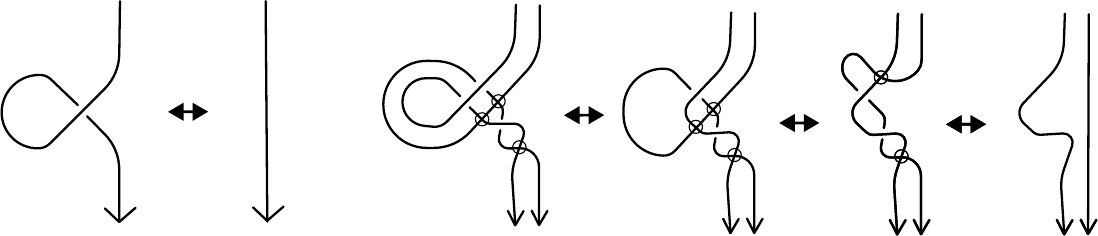}
}
\vspace{0.2cm}
\centerline{
\includegraphics[width=8cm]{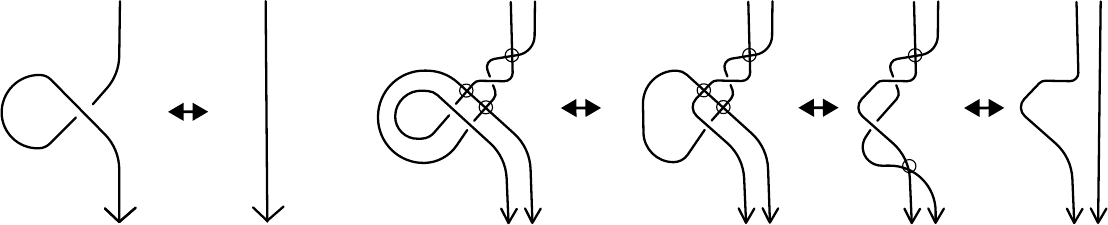}
}
\caption{}
\label{fig:proofR1}
\end{figure}

\begin{figure}[h]
\centerline{
\includegraphics[width=8cm, height=2cm]{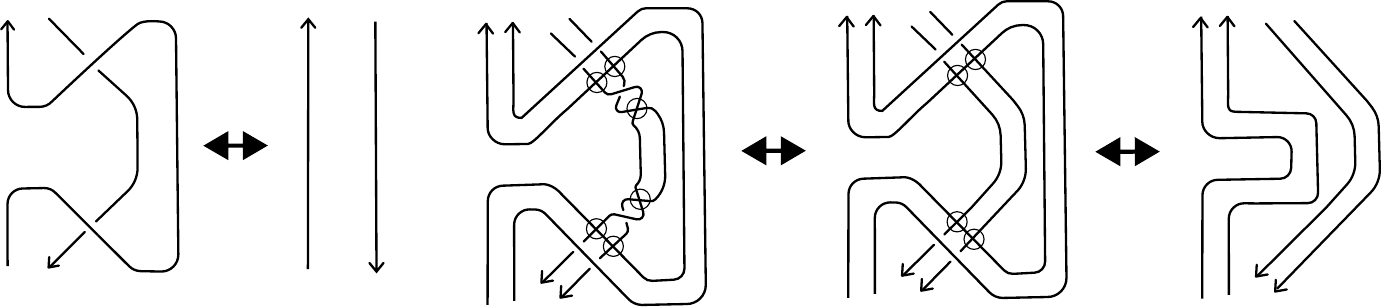}
}
\vspace{0.2cm}
\centerline{
\includegraphics[width=8cm, height=2cm]{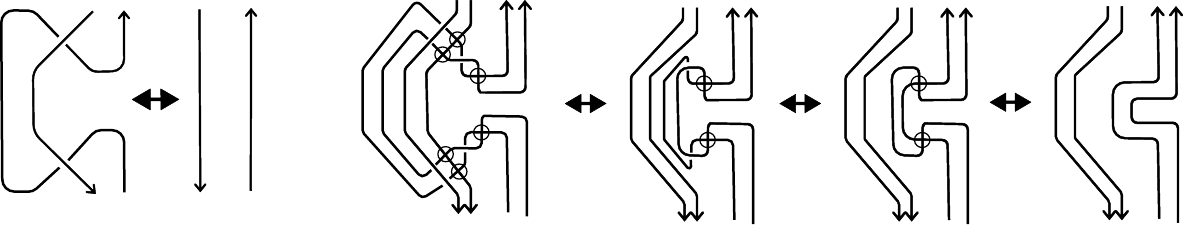}
}
\caption{}
\label{fig:proofR2}
\end{figure}

\begin{figure}[h]
\centerline{
\includegraphics[width=10cm]{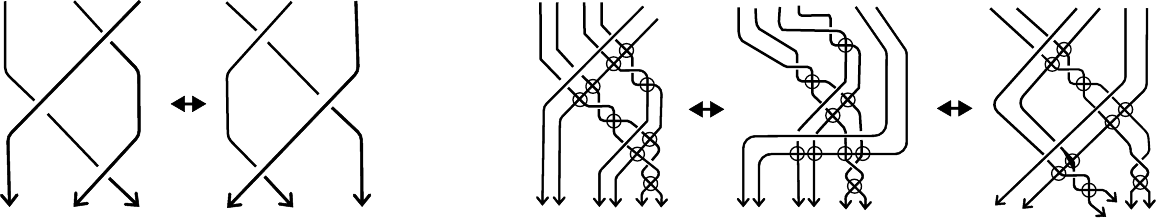}
}
\caption{}
\label{fig:proofR3}
\end{figure}

It is clear that
if $D$ and $D'$ are related by any of the virtual Reidemeister moves or welded moves,
then $\phi_+(D)$ and $\phi_+(D')$ are equivalent as welded links.
By the same argument, we also see that
$\phi_-(D)$ and $\phi_-(D')$ are equivalent as welded links.
\end{proof}

%%%%%%%%%%%%%%%%%%%

\section{On Alexander numberings and checkerboard colorings}

Let $D$ be a virtual link diagram.
An \emph{Alexander numbering} of $D$ is an assignment of elements of $\mathbb{Z}$ to all arcs of $D$, as illustrated in Figure~\ref{fig:AlexanderNumbering}. 

\begin{figure}[h]
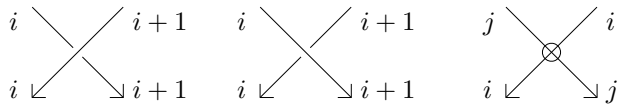

\centerline{
\hspace{0cm}\alexnum{.6mm}\hspace{.6cm}\alexnumv{.6mm}
}
\caption{The condition for an Alexander numbering}
\label{fig:AlexanderNumbering}
\end{figure}

Every classical link diagram admits an Alexander numbering.
However, there exist virtual link diagrams that do not. 
A virtual link or welded link is said to be \emph{almost classical} if it can be represented by a diagram that admits an Alexander numbering (cf. \cite{rboden, rsilver}).

\begin{thm}\label{thm:AlexanderNumbering}
For any virtual link diagram $D$, the parallel diagram $\phi_-(D)$ admits an Alexander numbering.
Thus, the welded link $\phi_-(D)$ is almost classical.
\end{thm}

\begin{proof}
Let $i$ be any fixed integer.
Near each classical crossing of $D$, assign integers to the arcs as shown in Figure~\ref{fig:proofAlexander}.
Assign the integer $i$ to all remaining arcs.
This produces an Alexander numbering for $\phi_-(D)$.
\end{proof}

\begin{figure}[h]
\centerline{
\includegraphics[width=10cm]{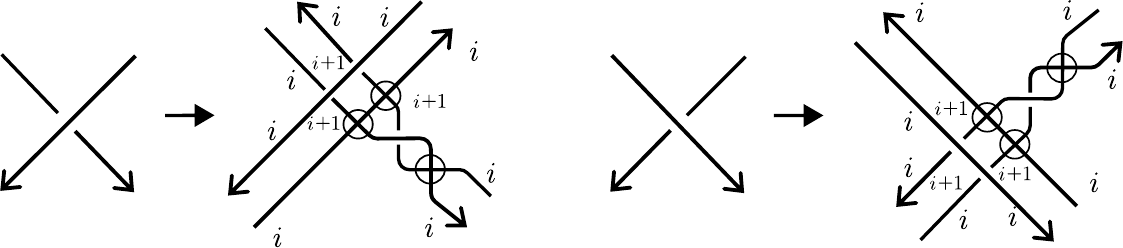}
}
\caption{}
\label{fig:proofAlexander}
\end{figure}

A \emph{mod-$2$ Alexander numbering} of a diagram $D$ is an assignment of elements of $\mathbb{Z}/2\mathbb{Z}$ to all arcs of $D$, again as illustrated in Figure~\ref{fig:AlexanderNumbering}.

A virtual or welded link is said to be \emph{checkerboard colorable}
or \emph{mod-$2$ almost classical} \cite{rCSWE}
if it can be represented by a diagram that admits a mod-$2$ Alexander numbering.
(The original definition of checkerboard colorability for virtual links, as given in \cite{rkn0}, requires that the virtual link be represented by an abstract link diagram \cite{rkk} that admits a checkerboard coloring.
This condition is equivalent to the existence of a virtual link diagram that admits a mod-$2$ Alexander numbering.)

\begin{cor}\label{cor:AlexanderNumbering}
For any virtual link diagram $D$, the parallel diagram $\phi_+(D)$ admits a mod-$2$ Alexander numbering.
Thus, the welded link $\phi_+(D)$ is checkerboard colorable.
\end{cor}

\begin{proof}
By Theorem~\ref{thm:AlexanderNumbering}, the parallel diagram $\phi_-(D)$ admits an Alexander numbering, and therefore it also admits a mod-$2$ Alexander numbering.
Because $\phi_+(D)$ and $\phi_-(D)$ differ only in orientation, a mod-$2$ Alexander numbering for $\phi_-(D)$ is also valid for $\phi_+(D)$.
\end{proof}

%%%%%%%%%%%%%%%%%%%
\section{On quandle colorings and the fundamental quandles}

A \emph{quandle} is a set $X$ equipped with a binary operation $*: X \times X \to X$ satisfying the following conditions \cite{rJoyce, rMatveev}:
\begin{itemize}
\item[(i)] For any $x \in X$, $x * x = x$.
\item[(ii)] For any $x, y \in X$, there exists a unique element $z \in X$ such that $z * y = x$.
\item[(iii)] For any $x, y, z \in X$, $(x * y) * z = (x * z) * (y * z)$.
\end{itemize}

The second condition may be replaced by the following \cite{rJoyce}:
\begin{itemize}
\item[(ii$'$)] There exists a binary operation $ \overline{*}: X \times X \to X$ such that for any $x, y \in X$,
$(x * y) \, \overline{*} \, y = x$ and $(x \, \overline{*} \, y) * y = x$.
\end{itemize}

An \emph{$X$-coloring} of a diagram $D$ is an assignment of elements of $X$ to the arcs of $D$ such that at each classical crossing the condition $a * b = c$ holds, where $a$, $b$, and $c$ are the elements assigned to the arcs as depicted in Figure~\ref{fig:xcoloring0}.
(Arcs passing through virtual crossings retain their assigned elements.)

\begin{figure}[h]
\centerline{
\includegraphics[width=4cm]{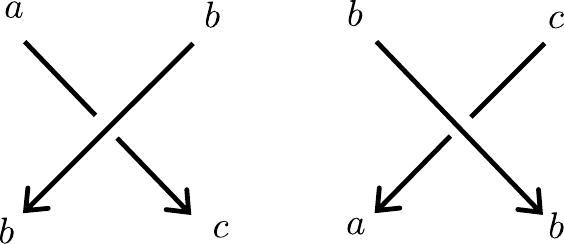}
}
\caption{The quandle coloring condition: $a * b = c$}
\label{fig:xcoloring0}
\end{figure}

An $X$-coloring is said to be \emph{trivial} if all arcs are assigned the same element of $X$. 

For a diagram $D$, let $\textrm{Col}_X(D)$ denote the set of all $X$-colorings of $D$. 

Fenn, Rim\'anyi, and Rourke \cite{FRR} investigated the fundamental racks and quandles of welded braids.
From their work, we obtain the following theorem (Kauffman proved it for virtual links in \cite{rkauD}):

\begin{thm}[Fenn, Rim\'anyi and Rourke \cite{FRR}]\label{thm:quandle}
If $D$ and $D'$ are equivalent as welded links, then
$\# \, \textrm{Col}_X(D) = \# \, \textrm{Col}_X(D')$.
\end{thm}

\begin{thm}
Assume $\# X \ge 2$.
For any welded link diagram $D$,
the parallel diagram $\phi_+(D)$ admits a nontrivial $X$-coloring.
The same holds for $\phi_-(D)$.
\end{thm}

\begin{proof}
See Figure~\ref{fig:xcoloring2} for the case of $\phi_+(D)$.
Reversing the orientation of the secondary subdiagram of $\phi_+(D)$ yields $\phi_-(D)$.
The coloring shown in Figure~\ref{fig:xcoloring2} still defines an $X$-coloring of $\phi_-(D)$.
\end{proof}

\begin{figure}[h]
\centerline{
\includegraphics[width=10cm]{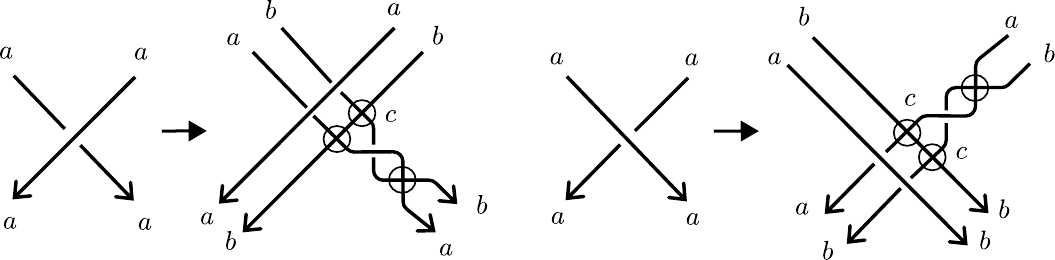}
}
\caption{A quandle coloring}
\label{fig:xcoloring2}
\end{figure}

\begin{thm}
Let $D$ be a diagram and $C$ an $X$-coloring of $D$ by a quandle $X$.
Then the parallel diagram $\phi_+(D)$ admits an $X$-coloring $\tilde{C}$ such that
the colors assigned to the arcs of the main subdiagram $\phi_+(D)_{\rm R}$ by $\tilde{C}$ coincide with those assigned to the corresponding arcs of $D$ by $C$.
The same holds for $\phi_-(D)$.
\end{thm}

\begin{proof}
See Figure~\ref{fig:xcoloringcbk} for the case of $\phi_+(D)$.
After reversing the orientation of the secondary subdiagram of $\phi_+(D)$, we obtain $\phi_-(D)$, and the coloring shown in Figure~\ref{fig:xcoloringcbk} again yields an $X$-coloring of $\phi_-(D)$.
\end{proof}

\begin{figure}[h]
\centerline{
\includegraphics[width=10cm]{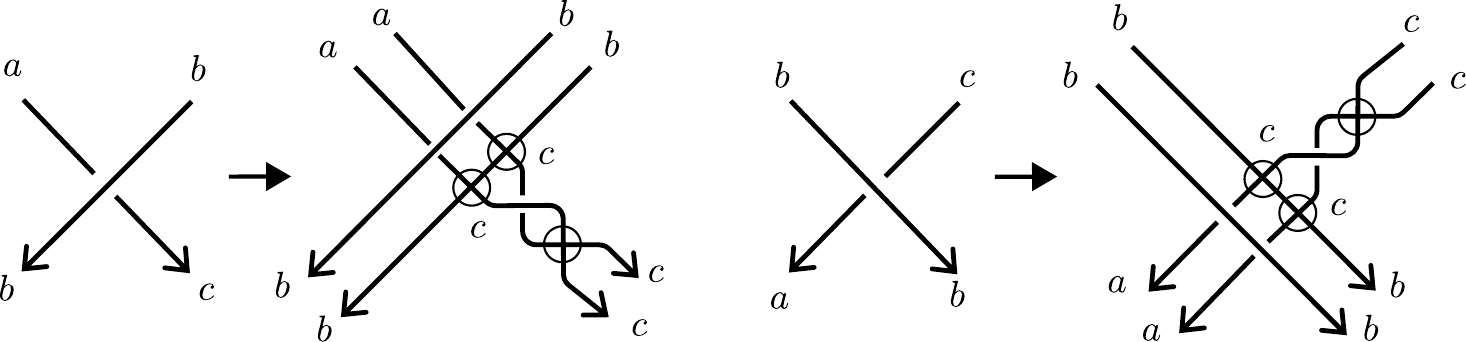}
}
\caption{The quandle coloring $\tilde{C}$ of $\phi_+(D)$}
\label{fig:xcoloringcbk}
\end{figure}

Assume that $D$ is a virtual knot diagram, that is, a one-component virtual link diagram.
Let
\[
\langle x_1, \dots, x_m  \mid  r_1, \dots, r_n \rangle_{\rm qdle}
\]
be a presentation of the fundamental quandle of $D$
(see \cite{rFR, rJoyce, rKamadaSeiichi2017} for such presentations). 

Let $L$ be a preferred longitude of $D$, represented by a word in ${x_1, \dots, x_m}$ with total exponent sum zero.
(It is regarded as an element of the associated group ${\rm As}(Q(D))$, naturally identified with the knot group $G(D)$ obtained from the Wirtinger presentation of $D$.)

\begin{thm}\label{thm:fundquandle}
Let $D$ be a virtual knot diagram with quandle presentation
\[
\langle x_1, \dots, x_m  \mid  r_1, \dots, r_n \rangle_{\rm qdle}, 
\]
and let $L$ be a preferred longitude.
Then the fundamental quandle of $\phi_+(D)$ has the presentation
\[
\langle x_1, \dots, x_m, y   \mid  r_1, \dots, r_n,   y^L =y \rangle_{\rm qdle}.
\]
\end{thm}

Here the notation $y^L$ follows the convention introduced by Fenn-Rourke \cite{rFR}.
For details, see \cite{rFR, rKamadaSeiichi2017}.

\begin{proof}
By Lemma~\ref{lem:twistpass}, we may assume without loss of generality that in $\phi_+(D)$ all $-1v$-twists or $v1$-twists occurring near classical crossings have been slid to a neighborhood of the terminal point of the longitude $L$ (as in the right-hand diagram of Figure~\ref{fig:exequiv1}).

Let $a_1, \dots, a_m$ denote the arcs of $D$.
The main subdiagram $\phi_+(D)_{\rm R}$ is a copy of $D$, so we use the same arc labels $a_1, \dots, a_m$ for it.
After the above deformation, the secondary subdiagram $\phi_+(D)_{\rm L}$ is parallel to the main subdiagram except near the terminal point of $L$; thus we label its arcs $a_1', \dots, a_m'$, with $a_i'$ parallel to $a_i$.

Let $x_1, \dots, x_m$ be the quandle elements corresponding to the arcs $a_1, \dots, a_m$ of $D$,
and use the same symbols for the corresponding elements in $Q(\phi_+(D))$.
The Wirtinger relations from the classical crossings of $D$ appear unchanged in $Q(\phi_+(D))$.
Let $y$ be the element corresponding to the arc at the initial (equivalently, terminal) point of $L$ in $\phi_+(D)_{\rm L}$.
As $y$ travels along $L$, when it returns to the terminal point it becomes $y^L$, yielding the relation $y^L = y$.
\end{proof}

\begin{exam}\label{exam:trefoilparallel}
Let $D$ be the diagram shown on the left of Figure~\ref{fig:exquandle1b}.

\begin{figure}[h]
\centerline{
\includegraphics[width=8cm]{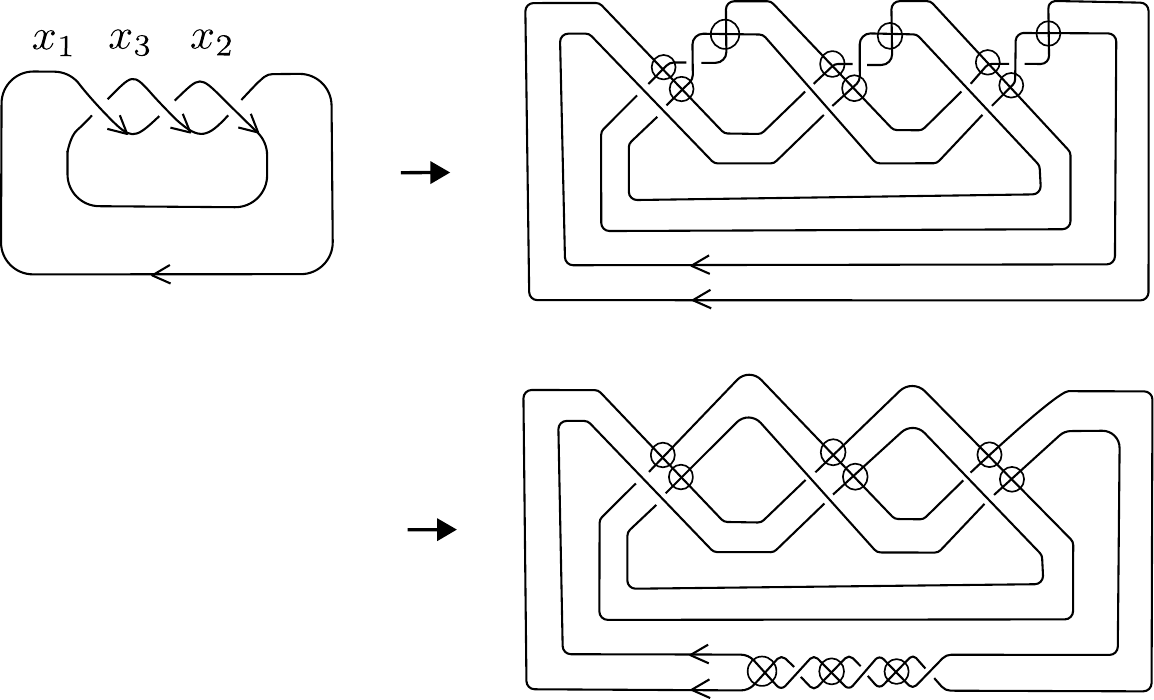}
}
\caption{}
\label{fig:exquandle1b}
\end{figure}

Then the fundamental quandle of $D$ has the presentation
\[
Q(D)=\langle x_1,x_2,x_3 \mid x_1^{x_3}=x_2, x_2^{x_1}=x_3, x_3^{x_2}=x_1\rangle_{\rm qdle},
\]
and preferred longitude $L$ is represented by $L=x_3x_1x_2x_1^{-3}$. 
Then we have 
\[
Q(\phi_+(D))=\langle x_1,x_2,x_3, y \mid x_1^{x_3}=x_2, x_2^{x_1}=x_3, x_3^{x_2}=x_1, y^L=y\rangle_{\rm qdle},
\] 
with $L=x_3x_1x_2x_1^{-3}$. 
\end{exam}

\begin{rmk}
This theorem extends to the case of a $\mu$-component virtual link diagram $D$.
If $L_1, \dots, L_{\mu}$ are chosen preferred longitudes, then the fundamental quandle of $\phi_+(D)$ has the presentation
\[
\langle x_1, \dots, x_m, y_1, \dots, y_\mu   \mid  r_1, \dots, r_n,   {y_i}^{L_i} =y_i (i=1, \dots, \mu) 
\rangle_{\rm qdle}.
\]
\end{rmk}

%%%%%%%%%%%%%%%%%%%%%%%%%%%

\section{Splitting Problem}

Let $D$ be a virtual link diagram, and let $\phi_+(D)$ be the parallel diagram of $D$.

At each classical crossing between $\phi_+(D)_{\rm R}$ and $\phi_+(D)_{\rm L}$,
the arc of $\phi_+(D)_{\rm R}$ passes over the arc of $\phi_+(D)_{\rm L}$.
A natural question to ask is whether 
$\phi_+(D) = \phi_+(D)_{\rm R} \cup \phi_+(D)_{\rm L}$ 
is equivalent to the disjoint union of $\phi_+(D)_{\rm R}$ and $\phi_+(D)_{\rm L}$. 

\begin{exam}
Let $D$ be the diagram depicted in Figure~\ref{fig:exhoplink1bk}.
Then $\phi_+(D)$ is not equivalent to the disjoint union of $\phi_+(D)_{\rm R}$ and $\phi_+(D)_{\rm L}$.
For the diagram 
$\phi_+(D) = \phi_+(D)_{\rm R} \cup \phi_+(D)_{\rm L}$, 
write $\phi_+(D)_{\rm R} = d_1 \cup d_2$ and
$\phi_+(D)_{\rm L} = d_1^\ast \cup d_2^\ast$.
Then the linking number $\textrm{lk}(d_1, d_1^\ast)$ is $-1/2$ $(\neq 0)$.
(Here the linking number $\textrm{lk}(D_1, D_2)$ of a 2-component diagram
$D = D_1 \cup D_2$ is defined as half of the sum of the signs of all classical crossings formed by arcs of $D_1$ and $D_2$.)

If $\phi_+(D)$ were equivalent to the disjoint union of $\phi_+(D)_{\rm R}$ and $\phi_+(D)_{\rm L}$, this number would have to be zero.
Thus, $\phi_+(D)$ is not equivalent to the disjoint union of $\phi_+(D)_{\rm R}$ and $\phi_+(D)_{\rm L}$.
\end{exam}

\begin{figure}[h]
\centerline{
\includegraphics[width=8.5cm]{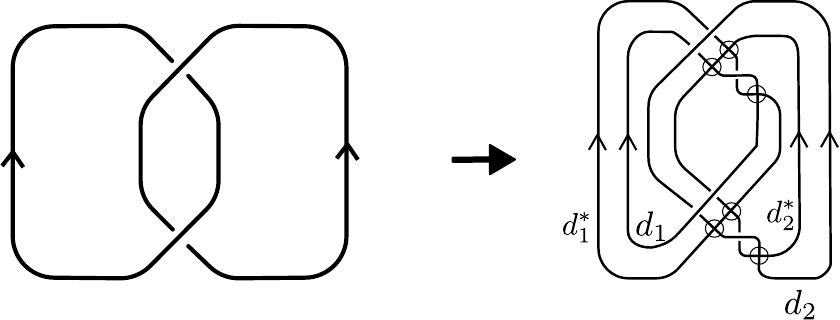}}
\centerline{\phantom{MMMM}$D$\phantom{MMMMMMMMMMMMMM} $\phi_+(D)$}
\caption{}
\label{fig:exhoplink1bk}
\end{figure}

As a corollary to Theorem~\ref{thm:fundquandle}, we obtain the following result.

\begin{prop}\label{prop:splitproblemqaundle}
Let $D$ be a virtual knot diagram, and let
\[
\langle x_1, \dots, x_m  \mid  r_1, \dots, r_n \rangle_{\rm qdle}
\]
be a presentation of its fundamental quandle.
Let $L$ be a preferred longitude.
If the fundamental quandle of $\phi_+(D)$, which has the presentation 
\[
\langle x_1, \dots, x_m, y   \mid  r_1, \dots, r_n,   y^L =y \rangle_{\rm qdle},
\]
is never isomorphic to the quandle
\[
\langle x_1, \dots, x_m, y   \mid  r_1, \dots, r_n \rangle_{\rm qdle},
\]
then 
$\phi_+(D) = \phi_+(D)_{\rm R} \cup \phi_+(D)_{\rm L}$
is not equivalent to the disjoint union of $\phi_+(D)_{\rm R}$ and $\phi_+(D)_{\rm L}$.
\end{prop}

\begin{proof}
This follows directly from Theorem~\ref{thm:fundquandle},
since the fundamental quandle of the disjoint union
of $\phi_+(D)_{\rm R}$ and $\phi_+(D)_{\rm L}$ has the presentation
$\langle x_1, \dots, x_m, y \mid r_1, \dots, r_n \rangle_{\rm qdle}$.
\end{proof}

\begin{prop}\label{prop:splitproblemgroup}
Let $D$ be a virtual knot diagram, and let
\[
\langle x_1, \dots, x_m  \mid  r_1, \dots, r_n \rangle
\]
be a presentation of its knot group.
Let $L$ be a preferred longitude.
If the knot group of $\phi_+(D)$, which has the presentation
\[
\langle x_1, \dots, x_m, y   \mid  r_1, \dots, r_n,  y L = L y  \rangle,
\]
is never isomorphic to the group
\[
\langle x_1, \dots, x_m, y   \mid  r_1, \dots, r_n \rangle,
\]
then
$\phi_+(D) = \phi_+(D)_{\rm R} \cup \phi_+(D)_{\rm L}$
is not equivalent to the disjoint union of $\phi_+(D)_{\rm R}$ and $\phi_+(D)_{\rm L}$.
\end{prop}

\begin{proof}
Since it is known that the associated group of the fundamental quandle of a diagram $D$
is isomorphic to the knot group of $D$,
the statement follows from Proposition~\ref{prop:splitproblemqaundle}.
\end{proof}

For example, for the diagram $D$ in Example~\ref{exam:trefoilparallel},
Proposition~\ref{prop:splitproblemgroup} implies that
$\phi_+(D) = \phi_+(D)_{\rm R} \cup \phi_+(D)_{\rm L}$
is not equivalent to the disjoint union of $\phi_+(D)_{\rm R}$ and $\phi_+(D)_{\rm L}$.
The authors thank Makoto Sakuma for pointing out this observation.
\noindent

%%%%%%%%%%%%%%%%%%%%%%%%%%%

\section*{Acknowledgement}

%{\bf Acknowledgement.}
This work was supported by JSPS KAKENHI Grant Numbers
23K03118 and 23H05437.

\end{document}